\documentclass [reqno]{amsart}
\usepackage{amssymb}
\usepackage{amsmath}
\usepackage{amsfonts}
\usepackage{graphicx}
\usepackage{subfig}
\usepackage{eurosym}
\usepackage{amssymb}
\usepackage{amsmath}
\usepackage{amsfonts}
\usepackage{graphicx}
\usepackage{subfig}

\newtheorem{theorem}{Theorem}
\theoremstyle{plain}

\newtheorem{definition}{Definition}

\numberwithin{equation}{section}

\begin{document}
\title[A review on the $V_{n}$-slant helices in lightlike cone $%
\mathbb{Q}
^{n+1}\subset E_{1}^{n+2}$]{A review on the $V_{n}$-slant helices in
lightlike cone $%
\mathbb{Q}
^{n+1}\subset E_{1}^{n+2}$}
\author{Fatma ALMAZ}
\address{department of mathematics, faculty of arts and sciences, batman
university, batman/ t\"{u}rk\.{ı}ye}
\email{fatma.almaz@batman.edu.tr}
\author{Handan \"{O}ztek\.{ı}n}
\address{department of mathematics, faculty of science, firat university,
elaz\i \u{g}/ t\"{u}rk\.{ı}ye}
\email{handanoztekin@gmail.com}
\thanks{This paper is in final form and no version of it will be submitted
for publication elsewhere.}

\begin{abstract}
In this paper, $V_{n}$-slant helices and the harmonic curvature functions of 
$V_{n}$-slant helices are defined in lightlike cone $%
\mathbb{Q}
^{n+1}$, and the differential equations of the harmonic curvature functions $%
H_{i}, 1 \leq i\leq n-2$ of $V_{n}$-slant helices are expressed by using
constant vector field $W $ that is the axis of $V_{n}$-slant helices. Also,
the necessary and sufficient conditions are given according to the condition
of being $V_{n}$-slant helices by using the asymptotic orthonormal frame in $%
\mathbb{Q}
^{n+1}$.
\end{abstract}

\maketitle

\section{Introduction}

In differential geometry, curves are fundamental objects for understanding
the local structure of a space. One of the simplest and most familiar curves
in three-dimensional Euclidean space is the helix. A helix is a curve whose
tangent vector makes a constant angle with a fixed direction. Slant helices
are a generalization of this concept and are defined by the principal normal
vector making a constant angle with a fixed direction vector. Also, slant
helices are curves with distinct geometric properties, such as a constant
curvature-to-torsion ratio. The geometric properties of slant helices can
play a role in describing or modelling the behaviour of various physical
systems: optical fibers in the shape of slant helices can provide ideal
geometric models for studying certain polarization rotation effects; the
motion of a charged particle in a constant magnetic field generally follows
a helical trajectory, the trajectory can be more complex, resulting in slant
helix-like shapes; in fluid dynamics, helical structures can be described
using slant helix geometry, while more complex helical structures can be
used to understand geometric models and physical properties when studying
the configurations of macromolecules, such as polymers and DNA, etc.

Studies of slant helices in pseudo-Riemann spacetime structures such as
lightlike cones offer significant potential future contributions to both
differential geometry and theoretical physics. Their fundamental importance
stems from their ability to extend familiar curve properties in classical
Euclidean geometry to specific spacetime regions with degenerate metrics
that describe the trajectories of light and massless particles. Such studies
have provided a deeper understanding of the geometric properties of curves
on lightlike cones (such as curvature, torsion, and the condition for being
"slant").

The geometry of curves on or associated with a lightlike cone (such as
curvatures and torsions) is affected by the lightlike cone's degenerate
metric structure. The slant helix condition is an additional geometric
property that restricts the movement of these curves on or within the
lightlike cone, for generalization of the slant helix definition to a
lightlike cone space, caution is necessary when generalizing the slant helix
definition in Euclidean space directly to a lightlike cone or to general
Lorentzian space, as the metric is not positive definite and the norm of
lightlike vectors is zero. However, the slant property (i.e., the fact that
a given vector of a curve makes a constant pseudo-angle with a fixed
direction or vector field) can be defined using different definitions of the
normal or versions of the Frenet frame adapted to Lorentzian space.
Examining slant helices in the context of a lightlike cone can shed light on
topics such as the curves on the lightlike cone can represent the
trajectories of massless particles lightlike, the slant helix nature of
these trajectories indicates that such trajectories have certain geometric
constraints; light and electromagnetic waves travel along lightlike
trajectories, slant helices on the lightlike cone can provide geometric
models for the trajectories of electromagnetic waves with specific
polarization properties or wave propagation patterns; the lightlike cone is
used to describe important structures in spacetime, such as the event
horizons of black holes, the geometry of the curves on these structures is
important for understanding the behaviour of gravity and fields.

In summary, the slant helix concept can be extended to Lorentzian space, and
in particular to the lightlike cone, by generalizing its metric and
providing appropriate definitions. In this context, the slant helix
properties of curves on or associated with the lightlike cone (such as
lightlike orbits) provide additional information about the geometry of these
orbits and may be of theoretical importance in modelling physical phenomena
such as light propagation, electromagnetic waves, and certain structures in
curved spacetime. The relationship hinges on how the geometric properties of
the curve (the slant helix condition) are expressed within the specific
structure of the lightlike cone and the physical implications of this.

A lot of mathematicians have worked on helices and slant curves. In \cite{13}%
, the authors express null generalized slant helices in Lorentzian space. In 
\cite{2}, non-null magnetic curves are studied by the authors in the null
cone. In \cite{3}, the notion of the involute-evolute curves for the curves
lying the surfaces are examined by the authors in Minkowski 3-space $%
E_{1}^{3}$. In \cite{4}, the helix and slant helices are investigated using
non-degenerate curves in term of Sabban frame in de Sitter $3$-space or Anti
de Sitter $3$-space $M^{3}\left( \delta _{0}\right) $, the necessary and
sufficient conditions are given for the non-degenerate curves to be slant
helix. By considering the ideas discussed in \cite{6} for defining a
general spiral on $3D$-Lorentzian backgrounds by taking into account the
work of Barros in \cite{5}, in \cite{17} slant helices are defined on the $%
3D $ sphere and in anti-De Sitter space. The term \textquotedblleft slant
helix\textquotedblright\ are introduced by Izumiya and Takeuchi \cite{12},
slant helices are defined by the property that their principal normals make
a constant angle with a fixed direction, and the characterization of slant
helices given in \cite{12} are extended to the Minkowski space $M_{1}^{3}$
by Ali and L\'{o}pez \cite{1}. In \cite{7}, the authors characterize general
helices and slant helices in terms of their associated curves and give a
canonical method to construct them. In \cite{8}, some characterizations are given a general intrinsic geometry according to bishop fram by the authors in $ E^{3} $. In \cite{9}, non-null $ k- $slant helices are defined in Minkowski 3-space. In \cite{11}, a definition of harmonic
curvature functions in terms of $V_{n}$ are given and defined a new kind of
slant helix Minkowski $n-$space. In \cite{14}, the notions of helix and
slant helices are investigated in the lightlike cone by using the asymptotic
orthonormal frame, and some characterizations of helices and slant helices
are presented. In \cite{15,16}, the cone curves and cone curvature function
etc are studied by the authors. Also, their representations and some
examples of cone curves are given in Minkowski space. In \cite{19,20,21,22}, extensive studies on slant helices in different space forms have been discussed by the authors. In \cite{23}, the author give basic notations and equations by explaining the Frenet formula of a regular curve in a psuedo-Riemannian manifold, the definitions of proper curve and proper helix are expressed.

\section{Preliminaries}

The $m-$dimensional pseudo-Euclidean space $E_{q}^{m}$ is given with the
metric 
\begin{equation}
\ d(x,y)=\overset{m-q}{\underset{i=1}{\sum }}x_{i}y_{i}-\overset{m}{\underset%
{j=m-q+1}{\sum }}x_{j}y_{j},  \tag{2.1}
\end{equation}%
where $x=(x_{1},x_{2},...,x_{m}),y=(y_{1},y_{2,}...,y_{m})\in E_{q}^{m}$, $%
E_{q}^{m}$ \ is a flat pseudo-Riemannian manifold of signature $(m-q,q)$.
Let $M$ be a submanifold of $E_{q}^{m}$. If the pseudo-Riemannian metric $d$
of $E_{q}^{m}$ induces a pseudo- Riemannian metric $d$ (respectively, a
Riemannian metric, a degenerate quadratic form) on $M$, then $M$ is said to
be timelike(respectively, spacelike, degenerate) submanifold of $E_{q}^{m}$.
Let $c$ be a fixed point in $E_{q}^{m}$ and $r>0$ be a constant. The
pseudo-Riemannian null cone is given by 
\begin{equation*}
\mathbf{%
\mathbb{Q}
}_{q}^{n}(c,r)=\{\varkappa \in E_{q}^{n+1}:d(\varkappa -c,\varkappa -c)=0\}
\end{equation*}%
and $\mathbf{%
\mathbb{Q}
}_{q}^{n}(c,r)$ is a degenerate hypersurface in $E_{q}^{n+1}$. The point $c$ is called the center
of $\mathbf{%
\mathbb{Q}
}_{q}^{n}(c,r)$. When $c=0$ and $q=1$, one simply denote $\mathbf{%
\mathbb{Q}
}_{1}^{n}(0)$ by $\mathbf{%
\mathbb{Q}
}^{n}$ and call it the lightlike cone (or simply the light cone), \cite%
{10,13,18}.

Let us recall from \cite{10,13,18} the definition of the Frenet frame and
curvatures, and let $E_{1}^{n+2}$ be the lightlike cone $(n+2)-$space $%
\mathbf{%
\mathbb{Q}
}^{n+1}$ $\subset $ $E_{1}^{n+2}$. A vector $\varkappa \neq 0$ in $%
E_{1}^{n+2}$ is called spacelike, timelike or null, if $\left\langle
\varkappa ,\varkappa \right\rangle >0$, $\left\langle \varkappa ,\varkappa
\right\rangle <0$ or $\left\langle \varkappa ,\varkappa \right\rangle =0$,
respectively. A frame field $\{E_{1},E_{2},...,E_{n+1},E_{n+2}\}$ on $%
E_{1}^{n+2}$ is called as asymptotic orthonormal frame field, if 
\begin{eqnarray*}
\left\langle E_{n+1},E_{n+1}\right\rangle  &=&\left\langle
E_{n+2},E_{n+2}\right\rangle =0,\left\langle E_{n+1},E_{n+2}\right\rangle =1,
\\
\left\langle E_{n+1},E_{i}\right\rangle  &=&\left\langle
E_{n+2},E_{i}\right\rangle =0,\ \left\langle E_{i},E_{j}\right\rangle
=\delta _{ij},\ i,j=1,2,...,n
\end{eqnarray*}%
and a frame field $\{E_{1},E_{2},...,E_{n+1},E_{n+2}\}$ on $E_{1}^{n+2}$ is
called a pseudo orthonormal frame field, if 
\begin{eqnarray*}
\left\langle E_{n+1},E_{n+1}\right\rangle  &=&-\left\langle
E_{n+2},E_{n+2}\right\rangle =1,\left\langle E_{n+1},E_{n+2}\right\rangle =0,
\\
\left\langle E_{n+1},E_{i}\right\rangle  &=&\left\langle
E_{n+2},E_{i}\right\rangle =0,\ \left\langle E_{i},E_{j}\right\rangle
=\delta _{ij},\ i,j=1,2,...,n
\end{eqnarray*}%
and let the curve $\varkappa :I\longrightarrow 
\mathbb{Q}
^{n+1}\subset E_{1}^{n+1}$, $t\longrightarrow \varkappa (t)\in 
\mathbb{Q}
^{n+1}$ be a regular curve in $%
\mathbb{Q}
^{n+1}$ and $\varkappa ^{\prime }(t)=\frac{d\varkappa (t)}{dt}$, for $%
\forall t\in I\subset 
\mathbb{R}
$ since $\left\langle \varkappa ,\varkappa \right\rangle =0$ and $%
\left\langle \varkappa ,d\varkappa \right\rangle =0$, $d\varkappa (t)$ is
spacelike. Then, the induced arclength (or simply the arclength) $s$ of the
curve $\varkappa (t)$ can be defined by $ds^{2}=\left\langle d\varkappa
(t),d\varkappa (t)\right\rangle $, \cite{10,16,18}.

For the arc length $s$ of the curve $\varkappa (s)=\varkappa (t(s))$ as the
parameter, and $\varkappa ^{\prime }(s)=V_{1}(s)=\frac{d\varkappa }{ds}$ is
a spacelike unit tangent vector field of the curve $\varkappa (s)$. Then,
for the vector field $y(s)$ the spacelike normal space $V^{n-1}$ of the
curve $\varkappa (s)$ such that they satisfy the following conditions:%
\textit{\ }%
\begin{eqnarray*}
\left\langle \varkappa (s),y(s)\right\rangle &=&1,\left\langle \varkappa
(s),\varkappa (s)\right\rangle =\left\langle y(s),y(s)\right\rangle
=\left\langle V_{1}(s),y(s)\right\rangle =0, \\
V^{n-1} &=&\{span\{\varkappa ,y,V_{1}\}\}^{\perp },span_{R}\{\varkappa
,y,V_{1},V^{n-1}\}=E_{1}^{n+2}.
\end{eqnarray*}

Therefore, we have the following Frenet formulas 
\begin{eqnarray*}
\varkappa ^{\prime }(s) &=&V_{1}(s) \\
\nabla _{V_{1}}V_{1} &=&\ V_{1}^{\prime }(s)=\kappa _{1}(s)\varkappa
(s)-y(s)\ +\tau _{1}(s)V_{2}(s) \\
\ \nabla _{V_{1}}V_{2} &=&V_{2}^{\prime }(s)=\kappa _{2}(s)\varkappa
(s)-\tau _{1}(s)V_{1}(s)\ +\tau _{2}(s)V_{3}(s) \\
\nabla _{V_{1}}V_{3} &=&\ V_{3}^{\prime }(s)=\kappa _{3}(s)\varkappa
(s)-\tau _{2}(s)V_{2}(s)\ +\tau _{3}(s)V_{4}(s) \\
&&....
\end{eqnarray*}%
\begin{equation}
\nabla _{V_{1}}V_{i}=V_{i}^{\prime }(s)=\kappa _{i}(s)\varkappa (s)-\tau
_{i-1}(s)V_{i-1}(s)\ +\tau _{i}(s)V_{i+1}(s)  \tag{2.2}
\end{equation}%
\begin{eqnarray*}
&&\text{ \ \ \ \ \ \ \ }.... \\
\nabla _{V_{1}}V_{n-1} &=&V_{n-1}^{\prime }(s)=\kappa _{i-1}(s)\varkappa
(s)-\tau _{n-2}(s)V_{n-2}(s)\ +\tau _{n-1}(s)V_{n}(s) \\
\nabla _{V_{1}}V_{n} &=&V_{n}^{\prime }(s)=\kappa _{n}(s)\varkappa (s)-\tau
_{n-1}(s)V_{n-1}(s) \\
y^{\prime }(s) &=&-\underset{i}{\overset{n}{\sum }}\kappa _{i}(s)V_{i}(s),
\end{eqnarray*}%
where $V_{2}(s)$,$V_{3}(s)$, ...,$V_{n}(s)\in V^{n-1}$, $\left\langle
V_{i},V_{j}\right\rangle =\delta _{ij}$, $i,j=2,...,n$. The functions $%
\kappa _{1}(s)$, $\kappa _{2}(s)$, ... , $\kappa _{n}(s)$, $\tau _{1}(s)$ ,
... , $\tau _{n-1}(s)$ are called cone curvature functions of the curve $%
\varkappa (s)$. The frame%
\begin{equation}
\{\varkappa (s),V_{1}(s),V_{2}(s),V_{3}(s),...,V_{n}(s),y(s)\}  \tag{2.3}
\end{equation}%
is called the asymptotic orthonormal frame on $E_{1}^{n+2}$ along the curve $%
\varkappa (s)$ in $%
\mathbb{Q}
^{n+1}$, \textit{\cite{10,16,18}}.

\section{Representation of $V_{n}-$slant helices in the lightlike cone $%
\mathbb{Q}
^{n+1}$}

In this section, $V_{n}-$slant helix curves in lightlike cone $%
\mathbb{Q}
^{n+1}$ are expressed and by using the harmonic curvatures $H_{i}, 1\leq i
\leq n-2$ some characterizations are given for $V_{n}-$slant helix curves
according to the asymptotic orthonormal frame in lightlike Cone $%
\mathbb{Q}
^{n+1}$.

\begin{definition}
Let $\varkappa :I\longrightarrow 
\mathbb{Q}
^{n+1}\subset E_{1}^{n+2}$ be a unit speed non-null curve with nonzero cone
curvature functions $\kappa _{1}(s)$, $\kappa _{2}(s)$, ... , $\kappa
_{n}(s) $; $\tau _{1}(s)$, ... , $\tau _{n-1}(s)$ in $%
\mathbb{Q}
^{n+1}$ and let $\{\varkappa
(s),V_{1}(s),V_{2}(s),V_{3}(s),...,V_{n}(s),y(s)\}$ be the asymptotic
orthonormal frame of the curve $\varkappa $. Then, the curve $\varkappa $ is
said to be as a $V_{n}-$slant helix in $%
\mathbb{Q}
^{n+1}$ if $n^{th}$ unit vector field $V_{n}$ makes a constant angle with a
fixed direction $W$, that is,%
\begin{equation}
\left\langle V_{n}(s),W\right\rangle =\eta _{n+1}=\text{constant, }\eta
_{n+1}\neq 0.  \tag{3.1}
\end{equation}

Therefore, for the subspace $Sp\{\varkappa
(s),V_{1}(s),V_{2}(s),V_{3}(s),...,V_{n}(s),y(s)\}$ the vector field $W$ can
be written as%
\begin{equation}
W\left( s\right) =\eta _{1}(s)\varkappa (s)+\eta _{n+2}(s)y(s)+\overset{n}{%
\underset{i=1}{\sum }}\eta _{i+1}(s)V_{i}(s).  \tag{3.2}
\end{equation}
\end{definition}

Let's try to express the vector field $W$ expressed in the definition above
again, so that one can find the differentiable functions $\eta _{i}$, $1\leq
i\leq n+2$ by using the equation (3.2). Thus, the following expressions can
be written 
\begin{equation}
\left\langle W,V_{i}\right\rangle =\eta _{i+1},1\leq i\leq n;\left\langle
W,y\right\rangle =\eta _{1};\left\langle W,x\right\rangle =\eta _{n+2}. 
\tag{3.3}
\end{equation}

First, if we take the derivative in the equation $\left\langle
W,y\right\rangle =\eta _{1}$, by using Frenet frame (2.2) and the equation
(3.2), one gets%
\begin{equation}
\eta _{1}^{\prime }=\left\langle W,y^{\prime }\right\rangle =-\underset{i}{%
\overset{n}{\sum }}\kappa _{i}\left\langle \eta _{1}\varkappa +\eta _{n+2}y+%
\overset{n}{\underset{j=1}{\sum }}\eta _{j+1}V_{j},V_{i}\right\rangle =-%
\underset{i}{\overset{n}{\sum }}\kappa _{i}\eta _{i+1}.  \tag{3.4}
\end{equation}

Secondly, if the derivative is taken in the equation $\left\langle
W,V_{i}\right\rangle =\eta _{i+1},1\leq i\leq n$, by using (2.2) and (3.2)
one gets%
\begin{equation*}
\eta _{i+1}^{\prime }=\left\langle W,V_{i}^{\prime }\right\rangle
=\left\langle \eta _{1}\varkappa +\eta _{n+2}y+\overset{n}{\underset{j=1}{%
\sum }}\eta _{j+1}V_{j},\kappa _{i}x-\tau _{i-1}V_{i-1}\ +\tau
_{i}V_{i+1}\right\rangle
\end{equation*}%
\begin{equation}
\eta _{i+1}^{\prime }=\kappa _{i}\eta _{n+2}-\tau _{i-1}\eta _{i}+\tau
_{i}\left( \eta _{i+2}+\eta _{1}\right) ,1\leq i\leq n.  \tag{3.5}
\end{equation}

Finally, by differentiating the equation $\left\langle W,\varkappa
\right\rangle =\eta _{n+2}$, \ from (2.2) and (3.2) one has%
\begin{equation}
\eta _{n+2}^{\prime }=\left\langle W,\varkappa ^{\prime }\right\rangle
=\left\langle \eta _{1}\varkappa +\eta _{n+2}y+\overset{n}{\underset{j=1}{%
\sum }}\eta _{j+1}V_{j},V_{1}\right\rangle =\eta _{2}.  \tag{3.6}
\end{equation}

Thus, from (3.4), (3.5) and (3.6) the following differential equations are
obtained%
\begin{equation}
\eta _{1} =-\underset{i}{\overset{n}{\sum }\int }\kappa _{i}\eta _{i+1}ds 
\tag{3.7a}
\end{equation}
\begin{equation}
\eta _{i+1} =\int \left( \kappa _{i}\eta _{n+2}-\tau _{i-1}\eta _{i}+\tau
_{i}\left( \eta _{i+2}+\eta _{1}\right) \right) ds;1\leq i\leq n  \tag{3.7b}
\end{equation}
\begin{equation}
\eta _{n+2} =\int \eta _{2}ds.  \tag{3.7c}
\end{equation}

If the equations given in (3.7) are taken into account in (3.2), the
following equation is obtained%
\begin{equation}
W\left( s\right) =-\left( \underset{i}{\overset{n}{\sum }\int }\kappa
_{i}\eta _{i+1}ds\right) \overrightarrow{\varkappa }(s)+\left( \int \eta
_{2}ds\right) \overrightarrow{y}(s)+\overset{n}{\underset{i=1}{\sum }}\eta
_{i+1}(s)\overrightarrow{V}_{i}(s).  \tag{3.8}
\end{equation}

\begin{definition}
Let $\varkappa :I\longrightarrow 
\mathbb{Q}
^{n+1}\subset E_{1}^{n+2}$ be a unit speed non-null curve with non zero
curvatures $\kappa _{i}(s)$, $\tau _{j}(s)$, $(i=1,2,...,n;j=1,2,...,n-1)$
in $%
\mathbb{Q}
^{n+1}$. Then, $\varkappa $ is a $V_{n}$-slant helix Harmonic curvature
functions $H_{i}:I\rightarrow IR$ are defined by%
\begin{eqnarray*}
H_{0} &=&0 \\
H_{1} &=&\frac{\tau _{1}}{\kappa _{2}},n=2 \\
H_{i} &=&\frac{1}{\tau _{i-1}}\left\{ 
\begin{array}{c}
\tau _{i}H_{i+2}+\kappa _{i}H_{1}-V_{1}[H_{i+1}] \\ 
+H_{i+1}\left( H_{n-2}-\tau _{1}H_{3}-\kappa _{1}H_{1}\right)%
\end{array}%
\right\} ,2<i<n.
\end{eqnarray*}
\end{definition}

\begin{proof}
Let $\varkappa :I\longrightarrow 
\mathbb{Q}
^{n+1}\subset E_{1}^{n+2}$ be a unit speed non-null curve with non zero
curvatures $\kappa _{i}(s)$, $\tau _{j}(s)$, $(i=1,2,...,n;$ $j=1,2,...,n-1)$
given as the asymptotic orthonormal frame $\{\varkappa
(s),V_{1}(s),V_{2}(s),V_{3}(s),...,V_{n}(s),y(s)\}$ in $%
\mathbb{Q}
^{n+1}$ and Harmonic curvature functions $H_{i}:I\rightarrow R$ of the curve 
$\varkappa $, $i=1,...,n-2$, and let $W$ be a unit constant vector field.
Then, since $\varkappa $ is a $V_{n}-$ slant helix, \ let's first try to
prove the accuracy of the equation given below 
\begin{equation}
\left\langle V_{i-2},W\right\rangle =H_{i-1}\left\langle
V_{1},W\right\rangle .  \tag{3.9}
\end{equation}

For $i=2$, since the unit constant vector field $W$, if the derivative is
taken the equation $V_{n}-$ slant helix, and cone Frenet formulas one gets 
\begin{equation*}
\left\langle D_{V_{1}}V_{i-2},W\right\rangle =V_{1}[H_{i-1}]\left\langle
V_{1},W\right\rangle +H_{i-1}\left\langle D_{V_{1}}V_{1},W\right\rangle
\end{equation*}%
\begin{equation}
\left\langle \kappa _{i-2}\varkappa -\tau _{i-3}V_{i-3}+\tau
_{i-2}V_{i-1},W\right\rangle =V_{1}[H_{i-1}]\left\langle
V_{1},W\right\rangle +H_{i-1}\left\langle \kappa _{1}\varkappa -y\ +\tau
_{1}V_{2},W\right\rangle  \tag{3.10}
\end{equation}%
and since $V_{n}-$ slant helix equation, one gets%
\begin{equation}
\left\langle \kappa _{n}\varkappa -\tau _{n-1}V_{n-1},W\right\rangle
=0\Rightarrow \left\langle V_{n-1},W\right\rangle =\frac{\kappa _{n}}{\tau
_{n-1}}\left\langle \varkappa ,W\right\rangle .  \tag{3.11}
\end{equation}

For $n=2,$ one has%
\begin{equation}
\left\langle V_{1},W\right\rangle =\frac{\kappa _{2}}{\tau _{1}}\left\langle
\varkappa ,W\right\rangle  \tag{3.12}
\end{equation}%
or 
\begin{equation}
\left\langle \varkappa ,W\right\rangle =H_{1}\left\langle
V_{1},W\right\rangle  \tag{3.13}
\end{equation}%
where $H_{1}=\frac{\tau _{1}}{\kappa _{2}}$. Furthermore, when $i$ is $i-1$,
the equation (3.9) becomes%
\begin{equation}
\left\langle V_{i-3},W\right\rangle =H_{i-2}\left\langle
V_{1},W\right\rangle ,  \tag{3.14}
\end{equation}%
when $i$ is $i+1$, the equation (3.9) becomes%
\begin{equation}
\left\langle V_{i-1},W\right\rangle =H_{i}\left\langle V_{1},W\right\rangle ,
\tag{3.15}
\end{equation}%
when $i$ is $4$, the equation (3.9) becomes%
\begin{equation}
\left\langle V_{2},W\right\rangle =H_{3}\left\langle V_{1},W\right\rangle , 
\tag{3.16}
\end{equation}%
when $i$ is $n+3$, the equation (3.9) becomes%
\begin{equation}
\left\langle V_{n+1},W\right\rangle =H_{n-2}\left\langle
V_{1},W\right\rangle ,  \tag{3.17}
\end{equation}%
where $V_{n+1}=y.$ From (3.13) and (3.14-3.17), the equation (3.10) is given
as 
\begin{eqnarray*}
\kappa _{i-2}\left\langle \varkappa ,W\right\rangle -\tau _{i-3}\left\langle
V_{i-3},W\right\rangle +\tau _{i-2}H_{i}\left\langle V_{1},W\right\rangle
&=&V_{1}[H_{i-1}]\left\langle V_{1},W\right\rangle \\
&&+H_{i-1}\kappa _{1}\left\langle \varkappa ,W\right\rangle \\
&&-H_{i-1}H_{n-2}\left\langle V_{1},W\right\rangle \\
&&+H_{i-1}\tau _{1}H_{3}\left\langle V_{1},W\right\rangle \text{ }
\end{eqnarray*}%
\begin{eqnarray*}
\tau _{i-3}\left\langle V_{i-3},W\right\rangle &=&\tau
_{i-2}H_{i}\left\langle V_{1},W\right\rangle +\kappa _{i-2}\left\langle
\varkappa ,W\right\rangle -V_{1}[H_{i-1}]\left\langle V_{1},W\right\rangle \\
&&+H_{i-1}H_{n-2}\left\langle V_{1},W\right\rangle -H_{i-1}\kappa
_{1}\left\langle \varkappa ,W\right\rangle -H_{i-1}\tau
_{1}H_{3}\left\langle V_{1},W\right\rangle
\end{eqnarray*}%
\begin{equation}
\left\langle V_{i-3},W\right\rangle =\frac{1}{\tau _{i-3}}\left\{ 
\begin{array}{c}
\tau _{i-2}H_{i}+\kappa _{i-2}H_{1}-V_{1}[H_{i-1}]+H_{i-1}H_{n-2} \\ 
-H_{i-1}\kappa _{1}H_{3}-H_{i-1}\kappa _{1}H_{1}%
\end{array}%
\right\} \left\langle V_{1},W\right\rangle .  \tag{3.18}
\end{equation}

Also, when $i$ is $i+1$, the equation (3.18) becomes%
\begin{equation}
\left\langle V_{i-2},W\right\rangle =\frac{1}{\tau _{i-2}}\left\{ 
\begin{array}{c}
\tau _{i-1}H_{i+1}+\kappa _{i-1}H_{1}-V_{1}[H_{i}] \\ 
+H_{i}H_{n-2}-H_{i}\tau _{1}H_{3}-H_{i}\kappa _{1}H_{1}%
\end{array}%
\right\} \left\langle V_{1},W\right\rangle  \tag{3.19}
\end{equation}%
\begin{equation}
\left\langle V_{i-2},W\right\rangle =H_{i-1}\left\langle
V_{1},W\right\rangle ,  \tag{3.20}
\end{equation}%
where 
\begin{equation}
H_{i-1}=\frac{1}{\tau _{i-2}}\left\{ 
\begin{array}{c}
\tau _{i-1}H_{i+1}+\kappa _{i-1}H_{1}-V_{1}[H_{i}] \\ 
+H_{i}H_{n-2}-H_{i}\tau _{1}H_{3}-H_{i}\kappa _{1}H_{1}%
\end{array}%
\right\} ,  \tag{3.21}
\end{equation}%
when $i$ is $i+1$, the equation (3.18) is given as%
\begin{equation}
H_{i}=\frac{1}{\tau _{i-1}}\left\{ \tau _{i}H_{i+2}+\kappa
_{i}H_{1}-V_{1}[H_{i+1}]+H_{i+1}\left( H_{n-2}-\tau _{1}H_{3}-\kappa
_{1}H_{1}\right) \right\} ,  \tag{3.22}
\end{equation}%
where $2<i<n+2.$
\end{proof}

\begin{theorem}
Let $\varkappa :I\longrightarrow 
\mathbb{Q}
^{n+1}\subset E_{1}^{n+2}$ be a unit speed non-null curve with non zero
curvatures $\kappa _{i}(s)$, $\tau _{j}(s)$, $(i=1,2,...,n;$ $j=1,2,...,n-1)$
given as the asymptotic orthonormal frame $\{\varkappa
(s),V_{1}(s),V_{2}(s),V_{3}(s),...,V_{n}(s),y(s)\}$ in $%
\mathbb{Q}
^{n+1}$ and let $H_{i}:I\rightarrow R,i=1,...,n-2$ be Harmonic curvature
functions of the curve $\varkappa $. If $\varkappa $ is a $V_{n}-$ slant
helix, a unit constant vector field $W$ is the axis of $V_{n}-$ slant helix
and $W$ is written as 
\begin{equation*}
W=(H_{1}\overrightarrow{y}+H_{n+2}\overrightarrow{\varkappa }+\underset{i=1}{%
\overset{n}{\sum }}\frac{1}{\tau _{i}}\left\{ 
\begin{array}{c}
\tau _{i+1}H_{i+3}+\kappa _{i+1}H_{1} \\ 
-V_{1}[H_{i+2}] \\ 
+H_{i+2}\left( H_{n-2}-\tau _{1}H_{3}-\kappa _{1}H_{1}\right)%
\end{array}%
\right\} \overrightarrow{V}_{i})\left\langle V_{1},W\right\rangle,
\end{equation*}%
where $H_{1}=\frac{\tau _{1}}{\kappa _{2}}.$
\end{theorem}

\begin{proof}
If the axis of $V_{n}-$ slant helix in $%
\mathbb{Q}
^{n+1}$ is $W$; then for the asymptotic orthonormal frame $\{\varkappa
(s),V_{1}(s),V_{2}(s),V_{3}(s),...,V_{n}(s),y(s)\}$ in $%
\mathbb{Q}
^{n+1}$ and by using the equation (3.9) one gets 
\begin{equation}
i =2;\left\langle \varkappa ,W\right\rangle =H_{1}\left\langle
V_{1},W\right\rangle  \tag{3.23a}
\end{equation}
\begin{equation}
3 \leq i\leq n+2;\left\langle V_{i-2},W\right\rangle =H_{i-1}\left\langle
V_{1},W\right\rangle  \tag{3.23b}
\end{equation}
\begin{equation}
i =n+3;\left\langle y,W\right\rangle =H_{n+2}\left\langle
V_{1},W\right\rangle ,  \tag{3.23c}
\end{equation}
where $\varkappa =V_{0\text{ }}$and $y=V_{n+1}.$ When the previous equations
and the equation (3.2) are considered together, the following expressions
are written 
\begin{equation}
\eta _{1} =\left\langle y,W\right\rangle =H_{n+2}\left\langle
V_{1},W\right\rangle  \tag{3.24a}
\end{equation}
\begin{equation}
\eta _{i} =\left\langle V_{i-1},W\right\rangle =H_{i}\left\langle
V_{1},W\right\rangle ;2\leq i\leq n+1  \tag{3.24b}
\end{equation}
\begin{equation}
\eta _{n+2} =\left\langle \varkappa ,W\right\rangle =H_{1}\left\langle
V_{1},W\right\rangle .  \tag{3.24c}
\end{equation}

Thus, when equations (3.24) are considered in (3.2), $W$ the equation is
easily obtained%
\begin{equation}
W=\left( H_{n+2}\overrightarrow{\varkappa }+H_{1}\overrightarrow{y}+\underset%
{i=1}{\overset{n}{\sum }}H_{i+1}\overrightarrow{V}_{i}\right) \left\langle
V_{1},W\right\rangle .  \tag{3.25}
\end{equation}

Finally, by using Harmonic curvature functions $H_{i}$ (3.12), (3.22) in
equation (3.25), one writes%
\begin{equation*}
W=(\frac{\tau _{1}}{\kappa _{2}}\overrightarrow{y}+H_{n+2}\overrightarrow{%
\varkappa }+\underset{i=1}{\overset{n}{\sum }}\frac{1}{\tau _{i}}\left\{ 
\begin{array}{c}
\tau _{i+1}H_{i+3}+\kappa _{i+1}H_{1}-V_{1}[H_{i+2}] \\ 
+H_{i+2}\left( H_{n-2}-\tau _{1}H_{3}-\kappa _{1}H_{1}\right)%
\end{array}%
\right\} \overrightarrow{V}_{i})\left\langle V_{1},W\right\rangle .
\end{equation*}
\end{proof}

\begin{theorem}
Let $\varkappa :I\longrightarrow 
\mathbb{Q}
^{n+1}\subset E_{1}^{n+2}$ be a unit speed non-null curve $%
\mathbb{Q}
^{n+1}$, $W$ be a unit constant vector field and for harmonic curvature
functions $H_{1},H_{2},...,H_{n-2}$ of the curve $\varkappa $. Then, if $%
\varkappa $ is a $V_{n}-$ slant helix, then the following equation satisfies%
\begin{equation*}
\frac{H_{n+1}^{2}}{\eta _{n+1}^{2}}-2H_{1}H_{n+2}=\underset{i=1}{\overset{n}{%
\sum }}H_{i+1}^{2}.
\end{equation*}
\end{theorem}

\begin{proof}
Let $\varkappa $ be a $V_{n}-$ slant helix with the arc length parameters.
Since $W$ is a unit constant vector field, by using (3.25) we obtain%
\begin{equation}
1=g\left( V_{1},W\right) ^{2}(2H_{n+2}H_{1}+\underset{i=1}{\overset{n}{\sum }%
}H_{i+1}^{2})  \tag{3.26}
\end{equation}%
and since $\varkappa $ is a $V_{n}-$ slant helix, 
\begin{equation}
\eta _{n+1}=\left\langle W,V_{n}\right\rangle =H_{n+1}\left\langle
W,V_{1}\right\rangle =\text{constant}  \tag{3.27}
\end{equation}%
and from (3.26), one gets%
\begin{equation}
\frac{H_{n+1}^{2}}{\eta _{n+1}^{2}}-2H_{n+2}H_{1}=\underset{i=1}{\overset{n}{%
\sum }}H_{i+1}^{2}.  \tag{3.28}
\end{equation}
\end{proof}

\begin{theorem}
Let $\varkappa :I\longrightarrow 
\mathbb{Q}
^{n+1}\subset E_{1}^{n+2}$ be a unit speed non-null curve $%
\mathbb{Q}
^{n+1}$, $W$ be a unit constant vector field, for the harmonic curvatures $%
H_{i},$ $i=1,2,...,n-2$ of $V_{n}-$ slant helix the following differential
equations are satisfied 
\begin{eqnarray*}
H_{1}^{\prime } &=&H_{2} \\
H_{2}^{\prime } &=&-H_{n+2}+k_{1}H_{1}+\tau _{1}H_{3} \\
H_{i+1}^{\prime } &=&k_{i}H_{1}-\tau _{i-1}H_{i}+\tau _{i}H_{i+2};2\leq
i\leq n \\
H_{n+2}^{\prime } &=&-\overset{n}{\underset{i=1}{\sum }}k_{i}H_{i+1}
\end{eqnarray*}%
and there exist functions $G_{i}\in C^{\infty }$ of satisfying equalities 
\begin{eqnarray*}
G_{1} &=&s+c;G_{2}=1;G_{3}=G_{n+2}-k_{1}G_{1} \\
G_{i+1} &=&\int (k_{i}G_{1}-\tau _{i-1}G_{i}+\tau _{i}G_{i+2})ds;3\leq i\leq
n \\
G_{n+1} &=&-\int (k_{1}+\overset{n}{\underset{i=2}{\sum }}k_{i}G_{i})ds
\end{eqnarray*}%
where 
\begin{equation*}
G_{i}(s)=\frac{H_{i}(s)}{H_{2}};1\leq i\leq n+2.
\end{equation*}
\end{theorem}

\begin{proof}
Suppose that the curve $\varkappa $ is a unit speed non-null $V_{n}-$ slant
helix in $%
\mathbb{Q}
^{n+1}$ and a unit constant vector field is $W$ given as the equation (3.2),
for the differential functions $\eta _{i},i=1,2,...,n+2$ and from the
asymptotic orthonormal frame $\{\varkappa
(s),V_{1}(s),V_{2}(s),V_{3}(s),...,V_{n}(s),y(s)\}$ and the harmonic
curvature functions $H_{i}$ of $\varkappa $, the position vector of $V_{n}-$
slant helix satisfies 
\begin{equation}
\eta _{1}=\left\langle y,W\right\rangle =H_{n+2}\eta _{2};\eta
_{i}=\left\langle V_{i-1},W\right\rangle =H_{i}\eta _{2};\eta
_{n+2}=\left\langle \varkappa ,W\right\rangle =H_{1}\eta _{2},  \tag{3.29}
\end{equation}%
where $\eta _{2}=\left\langle V_{1},W\right\rangle $. If the derivative of
the equation $\eta _{1}=\left\langle y,W\right\rangle =H_{n+2}\eta _{2}$ is
taken, one gets%
\begin{equation*}
H_{n+2}^{\prime }\eta _{2}=\left\langle y^{\prime },W\right\rangle
=\left\langle (H_{n+2}\overrightarrow{\varkappa }+H_{1}\overrightarrow{y}+%
\underset{j=1}{\overset{n}{\sum }}H_{j+1}\overrightarrow{V}_{j})\eta _{2},-%
\underset{i}{\overset{n}{\sum }}\kappa _{i}\overrightarrow{V}%
_{i}\right\rangle
\end{equation*}%
\begin{equation}
H_{n+2}^{\prime }=-\underset{i=1}{\overset{n}{\sum }}H_{i+1}\kappa _{i} 
\tag{3.30}
\end{equation}%
and by taking from the derivative of the equation $\eta _{n+2}=\left\langle
\varkappa ,W\right\rangle =H_{1}\eta _{2},$ one gets%
\begin{equation*}
H_{1}^{\prime }\eta _{2}=\left\langle \varkappa ^{\prime },W\right\rangle
=\left\langle \overrightarrow{V}_{1},(H_{n+2}\overrightarrow{\varkappa }%
+H_{1}\overrightarrow{y}+\underset{j=1}{\overset{n}{\sum }}H_{j+1}%
\overrightarrow{V}_{j})\eta _{2}\right\rangle
\end{equation*}%
\begin{equation}
H_{1}^{\prime }=H_{2},  \tag{3.31}
\end{equation}%
by taking from the derivative of the equation $\eta _{i}=\left\langle
V_{i-1},W\right\rangle =H_{i}\eta _{2},$ from (3.2) and (2.2) one gets%
\begin{equation*}
H_{i+1}^{\prime }\eta _{2}=\left\langle W,V_{i}^{\prime }\right\rangle =\eta
_{2}\left\langle H_{n+2}\overrightarrow{\varkappa }+H_{1}\overrightarrow{y}+%
\underset{j=1}{\overset{n}{\sum }}H_{j+1}\overrightarrow{V}_{j},\kappa _{i}%
\overrightarrow{\varkappa }-\tau _{i-1}\overrightarrow{V}_{i-1}\ +\tau _{i}%
\overrightarrow{V}_{i+1}\right\rangle
\end{equation*}%
\begin{equation}
H_{i+1}^{\prime }=\kappa _{i}H_{1}-\tau _{i-1}H_{i+1}+\tau
_{i}H_{i+2};j\rightarrow i-1,j\rightarrow i+1,  \tag{3.32}
\end{equation}
by taking from the derivative of the equation $\left\langle
V_{1},W\right\rangle =H_{2}\eta _{2},$ from (3.2) and (2.2) one gets%
\begin{equation*}
H_{2}^{\prime }\eta _{2}=\left\langle W,V_{1}^{\prime }\right\rangle
=\left\langle (H_{n+2}\overrightarrow{\varkappa }+H_{1}\overrightarrow{y}+%
\underset{j=1}{\overset{n}{\sum }}H_{j+1}\overrightarrow{V}_{j})\eta
_{2},\kappa _{1}\overrightarrow{\varkappa }-\overrightarrow{y}\ +\tau _{1}%
\overrightarrow{V}_{2}\right\rangle
\end{equation*}%
\begin{equation}
H_{2}^{\prime }=-H_{n+2}+H_{1}\kappa _{1}+\tau _{1}H_{3}\Rightarrow H_{3}=%
\frac{1}{\tau _{1}}\left( H_{2}^{\prime }+H_{n+2}-H_{1}\kappa _{1}\right) . 
\tag{3.33}
\end{equation}

Thus, the differential equations are obtained. Finally, the functions $H_{i}$
can be expressed in terms of $G_{i}$ functions as 
\begin{equation}
G_{i}(s)=\frac{H_{i}(s)}{H_{2}};1\leq i\leq n+2.  \tag{3.34}
\end{equation}%
\textbf{Case 1}: For $i=1$, $G_{1}(s)=\frac{H_{1}(s)}{H_{2}}$, and by
differentiating with respect to $s$ and applying (3.31), one gets%
\begin{equation}
G_{1}^{\prime }(s)=\frac{H_{1}^{\prime }(s)}{H_{2}}=\frac{H_{2}}{H_{2}}%
\Rightarrow G_{1}(s)=s+c.  \tag{3.35}
\end{equation}%
\textbf{Case 2}: For $i=2$, $G_{2}(s)=1$.\newline
\textbf{Case 3:} For $i=3$, $G_{3}(s)=\frac{H_{3}(s)}{H_{2}}$, and from $%
H_{3}=\frac{1}{\tau _{1}}\left( H_{2}^{\prime }+H_{n+2}-H_{1}\kappa
_{1}\right) $ one has 
\begin{equation}
G_{3}(s)=\frac{1}{\tau _{1}}\frac{H_{2}^{\prime }+H_{n+2}-H_{1}\kappa _{1}}{%
H_{2}}\Rightarrow G_{3}(s)=\frac{1}{\tau _{1}}G_{n+2}(s)-\frac{\kappa _{1}}{%
\tau _{1}}G_{1}(s)  \tag{3.36}
\end{equation}%
and from (3.34) one has%
\begin{equation}
G_{i+1}(s)=\frac{H_{i+1}(s)}{H_{2}};0\leq i\leq n+1.  \tag{3.37}
\end{equation}%
\textbf{Case 4}: For $2\leq i\leq n$ by differentiating (3.37) with respect
to $s$ and applying (3.32), one gets%
\begin{equation}
G_{i+1}^{\prime }=\frac{H_{i+1}^{\prime }}{H_{2}}=\frac{k_{i}H_{1}-\tau
_{i-1}H_{i}+\tau _{i}H_{i+2}}{H_{2}}=k_{i}G_{1}-\tau _{i-1}G_{i}+\tau
_{i}G_{i+2}  \tag{3.38}
\end{equation}%
\begin{equation}
G_{i+1}=\int \left( k_{i}G_{1}-\tau _{i-1}G_{i}+\tau _{i}G_{i+2}\right) ds. 
\tag{3.39}
\end{equation}%
\textbf{Case 5}: For $i=n+2$, $G_{n+2}(s)=\frac{H_{n+2}(s)}{H_{2}}$, and by
differentiating with respect to $s$ and applying (3.30), one gets%
\begin{equation}
G_{n+2}^{\prime }(s)=\frac{H_{n+2}^{\prime }(s)}{H_{2}}\Rightarrow
G_{n+2}(s)=-\int (\kappa _{1}+\underset{i=2}{\overset{n}{\sum }}G_{i}\kappa
_{i})ds.  \tag{3.40}
\end{equation}
\end{proof}

\section{Conclusion}

In this study, generalized $V_{n}-$slant helices are defined using newly
defined harmonic curvature functions in the $n+2$-dimensional lightlike cone
space $%
\mathbb{Q}
^{n+1}\subset E_{1}^{n+2}$, and the harmonic curvature functions of these
helices are expressed using the asymptotic orthonormal frame in $%
\mathbb{Q}
^{n+1}$. Furthermore, a constant vector field, denoted by $W$, the axis of
the helix, is defined along the curve using the $V_{n}-$ slant helices in $%
\mathbb{Q}
^{n+1}$. This constant vector field $W$ provides differential equations and
characterizations of $V_{n}-$ slant helices.

Studies of slant helices on lightlike cone lie at the interface between
curve theory, pseudo-Riemannian geometry, and theoretical physics. The
resulting geometric and physical insights provide a solid foundation for
future research, including modelling the behaviour of light and fields in
curved spacetime, understanding the geometric structure of fundamental
physical theories, and contributing to the solution of advanced mathematical
problems. Studies in this area have high potential to lead to new
discoveries, particularly at the intersection of gravity and
electromagnetism.

\section*{Funding}

Not applicable.

\section*{Informed Consent Statement}

Not applicable.

\section*{Conflicts of Interest}

The author declares no conflict of interest.

\end{document}